\documentclass[11pt]{article}
\usepackage{graphicx} 
\usepackage{amsmath}
\usepackage{amsfonts}
\usepackage{amsthm}
\usepackage{enumerate}
\usepackage[margin=0.8in]{geometry}
\usepackage{mlmodern}
\usepackage[T1]{fontenc}
\usepackage{authblk}

\usepackage{hyperref}   
\hypersetup{
    citecolor=magenta,
    colorlinks=true,
    linkcolor=blue,
    filecolor=magenta,  
    pdftitle={Overleaf Example},
    pdfpagemode=FullScreen,
}

\usepackage{enumitem}
\newcommand{\subscript}[2]{$#1 _ #2$}

\newtheorem{theorem}{Theorem}[section]
\newtheorem{lemma}[theorem]{Lemma}

\newtheorem{remark}[theorem]{Remark}
\theoremstyle{remark}

\newcommand{\eps}{\varepsilon}
\newcommand{\pa}{\partial}
\renewcommand{\epsilon}{\varepsilon}

\renewcommand{\O}{\mathcal{O}}

\providecommand{\keywords}[1]
{
  \small	
  \textbf{\textit{Keywords---}} #1
}
\title{Stabilization by Multiplicative It\^o Noise for Chafee-Infante Equation in Perforated Domains}
\author[1]{Hong Hai Ly}
\author[2]{Bao Quoc Tang}

\affil[1]{\small 
Department of Mathematics\\ 
University of Ostrava, Ostrava, Czech Republic\break
\href{mailto:hai.ly.s01@osu.cz}{hai.ly.s01@osu.cz}}
\affil[2]{\small Department of Mathematics and Scientific Computing, University of Graz, Graz, Austria\break  \href{mailto:quoc.tang@uni-graz.at}{quoc.tang@uni-graz.at}}
\date{}

\begin{document}

\maketitle
\begin{abstract}
	The stabilization by noise for parabolic equations in perforated domains, i.e. domains with small holes, is investigated. We show that when the holes are small enough, one can stabilize the unstable equations using suitable multiplicative It\^o noise. The results are quantitative, in the sense that we can explicitly estimate the size of the holes and diffusion coefficients for which stabilization by noise takes place. This is done by using the asymptotic behaviour of the first eigenvalue of the Laplacian as the hole shrinks to a point.
\end{abstract}
\keywords{Stabilization by noise, Perforated domains, Asymptotic behaviour of eigenvalues}
\section{Introduction and main result}
Consider a bounded domain $\O$ in $\mathbb{R}^n$, $n\in \{2,3\}$, with Lipschitz boundary $\partial\O$. Within this domain, there is a family of small pairwise disjoint open holes $\{E^i_\eps\}_{i=1}^{\ell}$, $\ell \ge 1$, which have size of order $\eps>0$, i.e. for each $E^i_\eps$, there is a ball $\mathbb{B}^i_\eps$ with radius $\eps$ such that $E^i_\eps \subset \mathbb{B}^i_\eps$. We assume that the holes are properly inside the domain, i.e. $\overline{E_\eps^i} \subset \O$, and the boundary $\partial E^i_\eps$ is also Lipschitz for each $i=1,\ldots, \ell$. The open perforated domain $\O_\eps= \O \backslash \bigcup\limits_{i=1}^{\ell}\overline{E^i_{\eps}}$ is obtained by removing these holes from $\O$. We study in this work a stochastic Chafee-Infante equation defined on this domain, described by the following
\begin{equation}\label{main-prob}
\begin{cases}
    du + (- \Delta u + u^3 - \beta u)dt = h(t,u)dW_t, &x\in\O_\varepsilon,\; t>0,\\
    u(x,t) = 0, &x\in\partial\O_\eps, \; t>0,\\
    u(x,0) = u_0(x), &x\in\O_\eps,
    \end{cases}
\end{equation}
where $\beta >0$ and $W_t$ is a standard real-valued scalar Wiener process defined on the probability space $(\Omega,\mathcal{F}, P)$ with natural filtration $(\mathcal{F}_t)_{t\ge 0}$, $dW_t$ denotes the It\^o differential and $u_0 \in L^2(\O_\eps)$ is given initial data. It holds obviously $\partial\O_\eps = \cup_{j=1}^\ell\partial E_\eps^i\cup \partial\O$.
     \medskip
    The diffusion coefficient $h: [0,\infty)\times \mathbb R \to \mathbb R$ is assumed to satisfy following conditions:
    \begin{enumerate} [label=(\subscript{H}{{\arabic*}})]
        \item\label{H_0} for all $t\ge 0$, $u\in \mathbb R$, $h(t,0) = 0$,
        \item \label{H_1} for all $t\ge 0$, $u\in \mathbb R$, $| h(t,u)|^2 \le \gamma(t) |u|^2$,  where $\gamma: [0,\infty) \to [0,\infty)$ is a continuous function such that
        \[ \limsup_{t \to \infty} \frac{1}{t} \int_0^t \gamma(s)ds \le \gamma_0  \quad \text{ with } \gamma_0 \text{ is a non-negative constant.}\] 
        \item \label{H_2} for $t\ge 0$, $u\in \mathbb R$, $|h(t,u)u|^2 \ge \rho(t) |u|^4$, where $\rho:[0,\infty) \to [0,\infty)$ is a  continuous function such that
         \[ \liminf_{t \to \infty} \frac{1}{t} \int_0^t \rho(s)ds \ge \rho_0 \quad \text{ with } \rho_0 \text{ is a non-negative constant.}\] 
    \end{enumerate}
    The assumption \ref{H_0} is to ensure that $u \equiv 0$ is a trivial steady state solution to \eqref{main-prob}. Conditions \ref{H_1} and \ref{H_2} are Lipschitz continuity at zero and dissipative properties.
Typical examples of the diffusion coefficient $h(t,u)$ satisfying conditions \ref{H_0}--\ref{H_2} include
\begin{itemize}
    \item the linear function $h(t,u)=\alpha u$, $\alpha \in \mathbb R$, for which $\gamma(t) = \rho(t) = \gamma_0 = \rho_0 = \alpha^2$ for all $t\in \mathbb R$,
    \item or the nonlinear function $h(t,u) = \dfrac{(2+t^2)(1+u^2)}{(1+t^2)(2+u^2)}u$ for which $\gamma(t)$ and $\rho(t)$ can be chosen as
    \begin{equation*}
        \gamma(t) = 16\left(\frac{2+t^2}{1+t^2}\right)^2, \quad \rho(t) = \left(\frac{2+t^2}{2(1+t^2)}\right)^2,
    \end{equation*}
    and consequently one can choose $\gamma_0 = 64$ and $\rho_0=1$.
\end{itemize}
Note that when $h(t,u)\equiv 0$, \eqref{main-prob} becomes the classical deterministic Chafee-Infante equation, for which the stability of the zero steady state is determined by the value of $\beta>0$ and the first eigenvalue of $-\Delta$ with homogeneous Dirichlet boundary condition.

\medskip
To put our work into context, let us briefly review the literature on stabilization by noise. The stabilization using multiplicative noise has been investigated since the eighties firstly for finite dimensional systems, see e.g. \cite{graham1982stabilization,arnold1983stabilization,arnold1990stabilization}. For the linear finite dimensional system $\dot{x} = Ax$, this phenomenon is completely characterized in \cite{arnold1983stabilization} where it was proved to be stabilizable by linear multiplicative noise if and only $\text{Trace}(A) <0$. The case of nonlinear systems and nonlinear noise, still in finite dimensional setting, is more complicated as it was shown in \cite{appleby2008stabilization} that random noise can also destabilize stable systems. The first work dealing with infinite dimensional system is \cite{kwiecinska1999stabilization}, therein the author showed that the multiplicative noise shifts the spectrum of Laplacian to the right, thus stabilizing the system. The stabilization by noise for partial differential equations has been then studied extensively, see e.g. \cite{caraballo2001stabilization,kwiecinska2002stabilization,caraballo2003stochastic,caraballo2004stabilisation,caraballo2007effect}. All these works showed that the is an infinite range of noise intensity that stabilizes the deterministic system. Notably, a recent work \cite{fellner2019stabilisation} showed that when the noise is acting solely on the boundary, one might stabilize the deterministic system only with a finite range of noise intensities. It is worthwhile to emphasize that, up to our knowledge, all the existing works considered solid domains, i.e. without holes.
On the other hand, the stabilization by \textit{deterministic} control in perforated domains has been also investigated by many authors since it appears frequently in applications, see e.g. \cite{zhikov1992asymptotic,cioranescu1991strong,cioranescu1994approximate,banks2011parameter,brown2022heat}. Motivated by these developments, we investigate in this paper, for the first time, the stabilizing effect of multiplicative noise in perforated domains for the Chafee-Infante equation \eqref{main-prob}. 

\medskip
The main result of this paper is the following theorem.
\begin{theorem}\label{theo:3.1}
    Under {\normalfont \ref{H_1}--\ref{H_2}}, there exists a unique global strong solution to the stochastic equation \eqref{main-prob}. 
    
    \medskip
    If {\normalfont \ref{H_0}} is fulfilled and
    \begin{equation*}
         \rho_0 - \frac{\gamma_0}{2} > - \lambda_1^D(\O) + \beta,
    \end{equation*}
    where $\lambda_1^D(\O)$ is the first eigenvalue of the Dirichlet-Laplacian on $\O$, 
    then there exists $\eps_0$ such that
    \begin{equation}\label{epsilon0}
        \eps_0 \sim 
        \begin{cases}
            e^{-\left(\lambda^D_1(\O)+\rho_0 -(\beta+{\gamma_0}/{2})\right)^{-1}} &\text{ for } n = 2,\\
            \lambda^D_1(\O)+\rho_0-(\beta+{\gamma_0}/{2}) &\text{ for } n = 3,
        \end{cases}
    \end{equation}
    and for all $0<\eps \le \eps_0$, the zero steady state solution to \eqref{main-prob} is globally exponentially stable almost surely, i.e.
    \begin{equation*}
        \|u(t)\|_{L^2(\O_\eps)} \le Ce^{-\delta t} \quad \mathbb{P}-\text{a.s.}
    \end{equation*}
    where the positive constants $C, \delta$ {\normalfont are independent of $\eps$}. Here $A\sim B$ means that $C_0A \le B \le C_1A$ for positive constants $C_0, C_1$ independent of $A$ and $B$.
\end{theorem}

\begin{remark}
	We remark that the stabilization shown in Theorem \ref{theo:3.1} is due to the multiplicative It\^o noise, which, after applying the It\^o formula, produces an ``artifical'' stabilizing term (see the proof in Section \ref{sec:3}. If the noise is considered in the sense of Stratonovich, such a term does not appear and the stabilizing effect is more intriguing. To obtain the stabilization by Stratonovich noise, following \cite{arnold1983stabilization,caraballo2004stabilisation,caraballo2007effect}, we will need to considered {\normalfont a collection of noisy terms} of the form $\sum_{i} B_iu\circ dW_t^i$ for suitable operators $B_i$. This interesting case is left for future research.
\end{remark}

It is emphasised that $\eps_0$ defined in \eqref{epsilon0} depends explicitly and only on the parameters $\beta, \rho_0, \gamma_0$ and the domain $\O$ through the first eigenvalue $\lambda_1^D(\O)$. It is known that when $\beta > \lambda_1^D(\O)$, the zero steady state to the deterministic equation \eqref{main-prob}, i.e. $h\equiv 0$, is unstable (see e.g. \cite[Theorem 2.2]{fellner2019stabilisation} for the case of Robin boundary conditions, whose arguments can be also applied to Dirichlet boundary conditions). By choosing the diffusion coefficient $h(t,u)$ such that \eqref{epsilon0} holds, for instance, $h(t,u) = \alpha u$ with $\alpha^2 > 2(-\lambda_1^D(\O) + \beta)$, then the zero solution is (exponentially) stabilized provided that the size $\eps$ of the holes do not exceed $\eps_0$ given in \eqref{epsilon0}. Our main idea is to use the asymptotic behavior of the first eigenvalue $\lambda_1^D(\O_\eps)$ of the Dirichlet-Laplacian in $\O_\eps$ as $\eps\to 0$, which will be proved in Section \ref{sec:2}. We then utilize these results in Section \ref{sec:3} to prove our main theorem.
\section{Asymptotic behavior of the first eigenvalue}\label{sec:2}
On the perforated domain $\O_\eps$, we investigate the eigenvalue problem
\begin{align}\label{EVPs}
\begin{cases}
-\Delta \varphi(x) = \lambda(\eps)\varphi(x), & x\in \O_\eps,\\
\varphi(x) = 0, & x\in \partial \O_\eps.
\end{cases}
\end{align}
Here, $\varphi(x)$ represents the eigenfunction, and $\lambda(\eps)$ is the corresponding eigenvalue. 

\medskip
Let $\lambda_1^D(\O_\eps)$ denote the first eigenvalue of $-\Delta$ in $\O_\eps$ under the Dirichlet boundary condition on $\partial \O_\eps$, and let $\lambda_1^D(\O)$ be the first eigenvalue of $-\Delta$ in $\O$ under the Dirichlet boundary condition on $\partial \O$. It is classical that
\begin{equation}\label{Poincare}
    \|\nabla u\|_{L^2(\O)}^2 \ge \lambda_1^D(\O)\|u\|_{L^2(\O)}^2 \; \text{ and } \; \|\nabla v\|_{L^2(\O_\eps)}^2 \ge \lambda_1^D(\O_\eps)\|v\|_{L^2(\O_\eps)}^2,
\end{equation}
for all $u\in H_0^1(\O)$ and all $v\in H_0^1(\O_\eps)$.

\medskip
The asymptotic behavior of $\lambda_1^D(\O_\eps)$ as $\eps\to 0$ has been extensively investigated in the literature, see e.g. \cite{ozawa1981singular,ozawa1982asymptotic,maz1985asymptotic,abatangelo2022ramification}.In the early works, the holes are often considered to be balls, which makes the analysis easier. In order to allow various shapes for the holes, later works utilize the concept of {\it capacity potential}, which measures the ability of the set $A$ to hold electric charge or conduct heat, see e.g. \cite{choquet1954theory}.
The capacity potential of a set $A \subset \O$ with respect to $\O$ is defined as follows:
\begin{align*}
    \textup{cap} (A) := \inf\left\{\int_\O |\nabla u|^2dx:\, u \in H^1_0(\O), \, u \ge 1 \text{ on } A\right\} .
\end{align*}

\begin{lemma}\label{lem1}
As $\eps\to 0$, the capacity potential $\textup{cap}(E^i_\eps)$ is estimated as
\begin{align*}
\textup{cap}(E^i_\eps) \le
\begin{cases}
\dfrac{2\pi}{-\log \eps} + o(\log^{-1}(\eps))& \text{if } n = 2,\\
4\pi\eps + o(\eps)& \text{if } n = 3.
\end{cases}
\end{align*}
\end{lemma}
\begin{proof}
    According to \cite{flucher1995approximation}, the capacity $\textup{cap}(\mathbb{B}^i_\eps)$ is given by:
  \begin{align*}
\textup{cap}(\mathbb{B}^i_\eps) =
\begin{cases}
\dfrac{2\pi}{-\log \eps} + o(\log^{-1}(\eps))& \text{if } n = 2,\\
4\pi\eps + o(\eps)& \text{if } n = 3.
\end{cases}
\end{align*}
 The monotonicity property of capacity \cite[Chapter II]{choquet1954theory}) states that the capacity of a smaller set is less than or equal to the capacity of any larger set containing it. Consequently, since $E^i_\eps \subseteq \mathbb{B}^i_\eps$, we have $\text{cap} (E^i_\eps) \le \text{cap}(\mathbb{B}^i_\eps)$, which gives the desired bound and finishes the proof of this Lemma.
\end{proof}

As $\eps\to 0$, each hole $E_\eps^i$ is assumed to shrink to a fixed point $x_i\in E_\eps^i$ for each $i=1,\ldots, \ell$. Now, we present the following theorem that establishes a relationship between the first eigenvalues of the perforated and original domains.
\begin{theorem}\label{theo:2.1}  The first eigenvalue of problem \eqref{EVPs} has the following asymptotics as $\eps$ tends to zero
\begin{align}\label{eq:asym_1}
\lambda^D_1(O_\eps) = \lambda^D_1(\O) +  \sum_{i=1}^\ell \varphi_1^2(x_i) \,\textup{cap}(E^i_\eps) + o(\textup{cap}(\cup_{i=1}^\ell E^i_\eps)),\quad \ell \in \mathbb{N},
\end{align}
where $\varphi_1$ represents the first eigenfunction corresponding to $\lambda^D_1(\O)$. Consequently
\begin{align}\label{asym_2}
    \lambda_1^D(\O_\eps)= 
    \begin{cases}
        \lambda_1^D(\O) + O((\log\eps)^{-1}) &\text{ for } n = 2,\\
        \lambda_1^D(\O) + O(\eps) &\text{ for } n = 3.\\
    \end{cases}
\end{align}
\end{theorem}
\begin{proof}
By applying \cite[Theorem 6]{flucher1995approximation} and using sub-additivity of the capacity and its asymptotic additivity for small disjoint sets,  we obtain the asymptotic expansion for each {\it simple} eigenvalue $\lambda_k(\O_\eps)$ as follows:
  \[ \lambda_k(\O_\eps) = \lambda_k(\O) + \sum_{i=1}^\ell \varphi_k^2(x_i) \textup{cap}(E^i_\eps) + o(\textup{cap}(\cup_{i=1}^\ell E^i_\eps)) \text{  as } \eps \to 0, \]
 where $\varphi_k$ are $L^2(\O
)$-orthonormal eigenfunctions associated with $\lambda_k(\O)$. It is classical that the first eigenvalue is simple, and therefore we have
\begin{align*}
\lambda^D_1(O_\eps) = \lambda^D_1(\O) + \sum_{i=1}^\ell \varphi_1^2(x_i) \textup{cap}(E^i_\eps) + o(\textup{cap}(\cup_{i=1}^\ell E^i_\eps)),
\end{align*}
which is the desired asymptotic expansion \eqref{eq:asym_1}. 

\medskip
To prove \eqref{asym_2}, it is sufficient to show that $\varphi_1$ is in $L^\infty(\O)$. Since the boundary $\pa\O$ is only Lipschitz continuous, using maximal regularity of elliptic equation for $-\Delta \varphi_1 = \lambda_1^D(\O)\varphi_1$ only gives $\varphi_1\in H^{3/2}(\O)$ which is not enough to obtain $L^\infty(\O)$-bound in the case $n=3$. Therefore, we use the results in 
\cite{davies1989heat} (see Theorem 2.1.4 and Example 2.1.8 therein), where it was shown that $\varphi_k$ are in $L^\infty(\O)$ for all $k\ge 1$. The asymptotic estimate \eqref{asym_2} the follows directly from \eqref{eq:asym_1}, Lemma \ref{lem1} and the boundedness of $\varphi_1$.
\end{proof}

\section{Stabilization by noise}\label{sec:3}

We start with the well-posedness for the stochastic equation \eqref{main-prob}.

\begin{theorem}[\protect{\cite[Theorem 3.1]{pardouxt1980stochastic}}]\label{thm:existence}
     Under assumptions \ref{H_1}--\ref{H_2}, the equation \eqref{main-prob} has a unique global strong solution in the sense 
     \begin{enumerate}[label=(\roman*)]
         \item $u \in L^2(\Omega; C([0,T];L^2(\O_\eps))$, and
         \item for $f(u)= u^3-\beta u$,
         \[
           u(t)+ \int_0^t (-\Delta u(s)+f(u(s)))ds=u_0+ \int_0^th(s,u(s)dW_s, \quad \;\text{a.e.}\; t\in (0,T), \;\text{a.s.}\; \omega\in\Omega.
         \]
     \end{enumerate}
\end{theorem}

We are now ready to prove our main result.
\begin{proof}[Proof of Theorem \ref{theo:3.1}]
    Fix $0\ne u_0 \in L^2(\O_\eps)$ and let $u(t)$ be the corresponding solution to problem \eqref{main-prob} with initial $u_0$ obtained in Theorem \ref{thm:existence}. 
    Applying It\^o's formula (see e.g. \cite[Theorem 4.3] {gawarecki2010stochastic} or, more general \cite[Theorem 6.10]{gawarecki2010stochastic}), for $\psi(u) := \log \|u(t)\|^2_{L^2(\O_\eps)}$, and using
    \begin{equation*}
       \psi'(u) = \frac{2 \langle u,\cdot \rangle_{L^2(\O_\eps)}}{\|u\|^2_{L^2(\O_\eps)}}, \quad \psi''(u)\xi = \frac{2 \langle \xi, \cdot\rangle_{L^2(\O_\eps)}}{\|u\|^2_{L^2(\O_\eps)}} -\frac{4\langle u,\xi\rangle_{L^2(\O_\eps)} \langle u,\cdot \rangle_{L^2(\O_\eps)}}{\|u\|^4_{L^2(\O_\eps)}} , \quad u,\xi \in L^2(\O_\eps).
    \end{equation*}
    we get
    \begin{equation} \label{eq:2}
    \begin{aligned}
        \log \|u(t)\|^2_{L^2(\O_\eps)}=&\log \|u_0 \|^2_{L^2(\O_\eps)} - \int_0^t \frac{2}{\|u(s)\|^2_{L^2(\O_\eps)}}\times \\
        &\times \left( \|\nabla u(s)\|^2_{L^2(\O_\eps)}+ \int_{\O_\eps}f(u(s))u(s)dx
        -\frac{1}{2} \|h(s,u(s))\|^2_{L^2(\O_\eps)} \right)ds  \\
        &-2 \int_0^t \frac{\langle u(s),h(s,u(s))\rangle^2_{L^2(\O_\eps)}}{\|u(s)\|^4_{L^2(\O_\eps)}}ds 
        +2 \int_0^t \frac{\langle u(s),h(s,u(s))\rangle_{L^2(\O_\eps)}}{\|u(s)\|^2_{L^2(\O_\eps)}}dW_s.  
        \end{aligned}
    \end{equation}
    Since $f(u) = u^3 - \beta u$, we have
    \begin{align}\label{eq:3}
        \int_{\O_\eps}f(u(s))u(s)dx \ge -\beta \|u(s)\|^2_{L^2(\O_\eps)}.
    \end{align}
    Moreover, thanks to the Poincar\'e inequality in \eqref{Poincare}
    \begin{align} \label{eq:4}
        \|\nabla u(s)\|^2_{L^2(\O_\eps)} \ge \lambda^D_1(\O_\eps)  \|u(s) \|^2_{L^2(\O_\eps)}.
    \end{align}
    Using assumption \ref{H_1}, we get
    \begin{align} \label{eq:5}
         \int_0^t \frac{\|h(s,u(s))\|^2_{L^2(\O_\eps)} }{\|u(s) \|^2_{L^2(\O_\eps)}} ds \le \int_0^t \gamma(s) ds.
    \end{align}
    Let $\delta>0$ be arbitrary.
    To deal with the stochastic integral, we utilize the exponential martingale inequality (see, e.g. \cite[ Lemma 1.1]{liu1997stability}) to get
    \begin{align*}
         \mathbb{P} \left\{\omega: \sup_{0\le t\le w}\left[ \int_0^t \frac{\langle u(s),h(s,u(s))\rangle_{L^2(\O_\eps)}}{\|u(s)\|^2_{L^2(\O_\eps)}} dW_s- \frac{\alpha}{2} \int_0^t \frac{\langle u(s),h(s,u(s))\rangle^2_{L^2(\O_\eps)}}{ \|u(s)\|^4_{L^2(\O_\eps)}} ds\right] > \frac{2\log k}{\alpha} \right\} \le \frac{1}{k^2}
    \end{align*}
where $0<\alpha<1$, $w=k \delta$ and $k=1,2,\dots$. We then apply the well-known Borel-Cantelli lemma to obtain that there exists an integer $k_0(\delta, \omega)>0$ for almost all $\omega \in \Omega$ such that
\begin{align}\label{eq:6}
    \int_0^t \frac{\langle u(s),h(s,u(s))\rangle_{L^2(\O_\eps)}}{\|u(s)\|^2_{L^2(\O_\eps)}} dW_s \le \frac{2 \log k}{\alpha} +\frac{\alpha}{2} \int_0^t \frac{\langle u(s),h(s,u(s))\rangle^2_{L^2(\O_\eps)}}{\|u(s)\|^4_{L^2(\O_\eps)}} ds
\end{align}
for all $0 \le t \le k \delta$, $k \ge k_0(\delta,\omega)$. Combining \eqref{eq:2}--\eqref{eq:6}, we see that there exists a positive random integer $k_1(\delta)$ such that almost surely
\begin{align*}
    \log \| u(t)\|^2_{L^2(\O_\eps)} \le \log \|u_0\|^2_{L^2(\O_\eps)} &+ \frac{4 \log k}{\alpha} + (2\beta -2 \lambda_1^D(\O_\eps))t \\
    &+ \int_0^t \gamma(s)ds -(2-\alpha)\int_0^t \frac{\langle u(s),h(s,u(s))\rangle^2_{L^2(\O_\eps)}}{\|u(s)\|^4_{L^2(\O_\eps)}}ds,
\end{align*}
for all $(k-1)\delta \le t \le k \delta$, $k \le k_1(\delta)$, which for the preceding $\delta >0 $ together with \ref{H_2} immediately implies that 
\begin{align*}
    \log \| u(t)\|^2_{L^2(\O_\eps)} \le \log \|u_0\|^2_{L^2(\O_\eps)}+ \frac{4 \log k}{\alpha} + (2\beta -2 \lambda_1^D(\O_\eps))t + \int_0^t \gamma(s)ds -(2-\alpha)\int_0^t \rho(s)ds.
\end{align*}
Thus, 
\begin{align*}
   \frac{1}{t} \log \| u(t)\|^2_{L^2(\O_\eps)} \le \frac{1}{t}\log \|u_0\|^2_{L^2(\O_\eps)}+ \frac{1}{t}\frac{4 \log k}{\alpha} + (2\beta -2 \lambda_1^D(\O_\eps)) -(2-\alpha) \frac{1}{t} \int_0^t \rho(s)ds + \frac{1}{t} \int_0^t \gamma(s) ds.
\end{align*}
Taking the limsup of both sides as $t\to\infty$, using $\limsup(-y_n) = -\liminf y_n$, and the assumptions \ref{H_1} and \ref{H_2}
\begin{align*}
    \limsup_{t \to \infty} \frac{1}{t} \log \|u(t)\|^2_{L^2(\O_\eps)} \le 2 \beta-2  \lambda_1^D(\O_\eps)-(2-\alpha)\rho_0+\gamma_0.
\end{align*}
Passing to the limit $\alpha \to 0$ yields
\begin{align*}
    \limsup_{t \to \infty} \frac{1}{t} \log \|u(t)\|^2_{L^2(\O_\eps)} \le 2\beta +\gamma_0-2 \lambda_1^D(\O_\eps) -2 \rho_0 \quad\text{ a.s.}
\end{align*}
Due to Theorem \ref{theo:2.1}, we have the asymptotic formula
\begin{align*}
    \lambda_1^D(\O_\eps)= 
    \begin{cases}
        \lambda_1^D(\O) + O((\log\eps)^{-1}) &\text{ for } n = 2,\\
        \lambda_1^D(\O) + O(\eps) &\text{ for } n = 3.\\
    \end{cases}
\end{align*}
Therefore, for all $\eps \le \eps_0$ with $\eps_0$ is in \eqref{epsilon0}, we arrive at the estimate
\begin{align*}
     \limsup_{t \to \infty} \frac{1}{t} \log \|u(t)\|^2_{L^2(\O_\eps)} \le 2\beta +\gamma_0-2 \lambda_1^D(\O) -2 \rho_0 \quad\text{ a.s.}
\end{align*}
It follows that for $\lambda_1^D(\O) + \rho_0> \beta+ \frac{1}{2} \gamma_0$, the zero solution becomes exponentially stable.
\end{proof}

\medskip
\noindent{\Large\textbf{Acknowledgement.}} 

\medskip
\noindent We thank the referees for the valuable comments and remarks. This work is carried out during an internship of Hong-Hai Ly at University of Graz, and the university's hospitality is acknowledged. Hong-Hai Ly is supported by grant SGS09/P\v{r}F/2023. B.T. received funding from the FWF project “Quasi-steady-state approximation for PDE”, number I-5213.


\begin{thebibliography}{00}
	
	\bibitem[ACW83]{arnold1983stabilization}
	Ludwig Arnold, Hans Crauel, and Volker Wihstutz.
	\newblock Stabilization of linear systems by noise.
	\newblock {\em SIAM Journal on Control and Optimization}, 21(3):451--461, 1983.
	
	\bibitem[ALM22]{abatangelo2022ramification}
	Laura Abatangelo, Corentin L{\'e}na, and Paolo Musolino.
	\newblock Ramification of multiple eigenvalues for the {D}irichlet-{L}aplacian
	in perforated domains.
	\newblock {\em Journal of Functional Analysis}, 283(12):109718, 2022.
	
	\bibitem[AMR08]{appleby2008stabilization}
	John~AD Appleby, Xuerong Mao, and Alexandra Rodkina.
	\newblock Stabilization and destabilization of nonlinear differential equations
	by noise.
	\newblock {\em IEEE Transactions on Automatic Control}, 53(3):683--691, 2008.
	
	\bibitem[Arn90]{arnold1990stabilization}
	Ludwig Arnold.
	\newblock Stabilization by noise revisited.
	\newblock {\em ZAMM-Journal of Applied Mathematics and Mechanics/Zeitschrift
		f{\"u}r Angewandte Mathematik und Mechanik}, 70(7):235--246, 1990.
	
	\bibitem[BCCW11]{banks2011parameter}
	HT~Banks, D~Cioranescu, AK~Criner, and WP~Winfree.
	\newblock Parameter estimation for the heat equation on perforated domains.
	\newblock {\em Journal of Inverse \& Ill-Posed Problems}, 19(6), 2011.
	
	\bibitem[BS22]{brown2022heat}
	Kymani~M Brown and Mohammad~R Shaeri.
	\newblock Heat transfer coefficients in perforated fins.
	\newblock In {\em Proceedings of the 9th International Conference on Fluid
		Flow, Heat and Mass Transfer (FFHMT’22)}, 2022.
	
	\bibitem[CCLR07]{caraballo2007effect}
	Tom{\'a}s Caraballo, Hans Crauel, Jos{\'e} Langa, and James Robinson.
	\newblock The effect of noise on the {C}hafee-{I}nfante equation: a nonlinear
	case study.
	\newblock {\em Proceedings of the American Mathematical Society},
	135(2):373--382, 2007.
	
	\bibitem[CDZ91]{cioranescu1991strong}
	Doina Cioranescu, Patrizia Donato, and Enrike Zuazua.
	\newblock Strong convergence on the exact boundary controllability of the wave
	equation in perforated domains.
	\newblock In {\em Estimation and Control of Distributed Parameter Systems:
		Proceedings of an International Conference on Control and Estimation of
		Distributed Parameter Systems, Vorau, July 8--14, 1990}, pages 99--113.
	Springer, 1991.
	
	\bibitem[CDZ94]{cioranescu1994approximate}
	Doina Cioranescu, Patrizia Donato, and Enrique Zuazua.
	\newblock Approximate boundary controllability for the wave equation in
	perforated domains.
	\newblock {\em SIAM journal on control and optimization}, 32(1):35--50, 1994.
	
	\bibitem[CGAR03]{caraballo2003stochastic}
	Tom{\'a}s Caraballo, Mar{\i}a~J Garrido-Atienza, and Jos{\'e} Real.
	\newblock Stochastic stabilization of differential systems with general decay
	rate.
	\newblock {\em Systems \& control letters}, 48(5):397--406, 2003.
	
	\bibitem[Cho54]{choquet1954theory}
	Gustave Choquet.
	\newblock Theory of capacities.
	\newblock In {\em Annales de l'institut Fourier}, volume~5, pages 131--295,
	1954.
	
	\bibitem[CLM01]{caraballo2001stabilization}
	Tom{\'a}s Caraballo, Kai Liu, and Xuerong Mao.
	\newblock On stabilization of partial differential equations by noise.
	\newblock {\em Nagoya Mathematical Journal}, 161:155--170, 2001.
	
	\bibitem[CR04]{caraballo2004stabilisation}
	Tom{\'a}s Caraballo and James~C Robinson.
	\newblock Stabilisation of linear pdes by stratonovich noise.
	\newblock {\em Systems \& control letters}, 53(1):41--50, 2004.
	
	\bibitem[Dav89]{davies1989heat}
	Edward~Brian Davies.
	\newblock {\em Heat kernels and spectral theory}.
	\newblock Number~92. Cambridge university press, 1989.
	
	\bibitem[Flu95]{flucher1995approximation}
	Martin Flucher.
	\newblock Approximation of {D}irichlet eigenvalues on domains with small holes.
	\newblock {\em Journal of mathematical analysis and applications},
	193(1):169--199, 1995.
	
	\bibitem[FSTT19]{fellner2019stabilisation}
	Klemens Fellner, Stefanie Sonner, Bao~Quoc Tang, and Do~Duc Thuan.
	\newblock Stabilisation by noise on the boundary for a {C}hafee-{I}nfante
	equation with dynamical boundary conditions.
	\newblock {\em Discrete and Continuous Dynamical Systems-B}, 24(8):4055--4078,
	2019.
	
	\bibitem[GM10]{gawarecki2010stochastic}
	Leszek Gawarecki and Vidyadhar Mandrekar.
	\newblock {\em Stochastic differential equations in infinite dimensions: with
		applications to stochastic partial differential equations}.
	\newblock Springer Science \& Business Media, 2010.
	
	\bibitem[GS82]{graham1982stabilization}
	R~Graham and A~Schenzle.
	\newblock Stabilization by multiplicative noise.
	\newblock {\em Physical Review A}, 26(3):1676, 1982.
	
	\bibitem[Kwi99]{kwiecinska1999stabilization}
	Anna~A Kwieci{\'n}ska.
	\newblock Stabilization of partial differential equations by noise.
	\newblock {\em Stochastic Processes and their Applications}, 79(2):179--184,
	1999.
	
	\bibitem[Kwi02]{kwiecinska2002stabilization}
	Anna Kwieci{\'n}ska.
	\newblock Stabilization of evolution equations by noise.
	\newblock {\em Proceedings of the American Mathematical Society},
	130(10):3067--3074, 2002.
	
	\bibitem[Liu97]{liu1997stability}
	Kai Liu.
	\newblock On stability for a class of semilinear stochastic evolution
	equations.
	\newblock {\em Stochastic processes and their applications}, 70(2):219--241,
	1997.
	
	\bibitem[MNP85]{maz1985asymptotic}
	Vladimir~Gilelevich Maz'ya, Sergei~Aleksandrovich Nazarov, and
	Boris~Alekseevich Plamenevskii.
	\newblock Asymptotic expansions of the eigenvalues of boundary value problems
	for the {L}aplace operator in domains with small holes.
	\newblock {\em Mathematics of the USSR-Izvestiya}, 24(2):321--345, 1985.
	
	\bibitem[Oza81]{ozawa1981singular}
	Shin Ozawa.
	\newblock Singular variation of domains and eigenvalues of the {L}aplacian.
	\newblock {\em Duke Math. J.}, 48(1):767--778, 1981.
	
	\bibitem[Oza82]{ozawa1982asymptotic}
	Shin Ozawa.
	\newblock An asymptotic formula for the eigenvalues of the {L}aplacian in a
	domain with a small hole.
	\newblock {\em Proc. Japan Acad. Ser. A Math. Sci.}, 58(5):5--8, 1982.
	
	\bibitem[Par80]{pardouxt1980stochastic}
	E~Pardouxt.
	\newblock Stochastic partial differential equations and filtering of diffusion
	processes.
	\newblock {\em Stochastics}, 3(1-4):127--167, 1980.
	
	\bibitem[Zhi92]{zhikov1992asymptotic}
	VV~Zhikov.
	\newblock Asymptotic problems connected with the heat equation in perforated
	domains.
	\newblock {\em Sbornik: Mathematics}, 71(1):125--147, 1992.
	
\end{thebibliography}
\end{document}